\documentclass[reqno,final,11pt]{amsart}

\usepackage[utf8]{inputenc}
\usepackage[T1]{fontenc}
\usepackage{amsmath,amssymb,amsfonts,amsthm}
\usepackage{geometry}
\usepackage{enumitem}
\usepackage{xcolor}
\usepackage{hyperref}
\usepackage{natbib}
\usepackage{lineno}
\usepackage{url}
\definecolor{linkcolor}{rgb}{0.16,0.42,0.70}
\hypersetup{colorlinks=true,linkcolor=linkcolor,citecolor=linkcolor,urlcolor=linkcolor}

\newcommand{\R}{\mathbb R}
\newcommand{\N}{\mathbb N}
\newcommand{\E}{\mathbb E}
\newcommand{\Pbb}{\mathbb P}
\newcommand{\Kk}{\mathbf K}
\newcommand{\one}{\mathbf 1}
\newcommand{\ind}[1]{\mathbf 1_{\{#1\}}}
\newcommand{\ip}[2]{\langle #1,#2\rangle}
\newcommand{\cM}{\mathcal M}

\theoremstyle{plain}
\newtheorem{theorem}{Theorem}[section]

\newtheorem{lemma}[theorem]{Lemma}

\theoremstyle{definition}

\newtheorem{example}[theorem]{Example}
\theoremstyle{remark}
\newtheorem{remark}[theorem]{Remark}

\numberwithin{equation}{section}

\title[Mixed Finite–Infinite Mean Multitype Branching Systems]{Convergence of a Critical Multitype Branching Particle System with Mixed Finite and Infinite Mean Lifetimes}
\author[J.H. Ram\'{i}rez-Gonz\'{a}lez]{ J.H. Ram\'{i}rez-Gonz\'{a}lez}
\address{Universidade de S\~ao Paulo, Brazil}
\email{hermenegildo.ramirez@usp.br}

\author{Fabio Prates Machado}

\email{fmachado@ime.usp.br}

\renewcommand{\subjclassname}{2020 Mathematics Subject Classification}



\begin{document}
% \linenumbers
\begin{abstract}
We study a critical multitype branching particle system in
\(\R^N\) with finite type space \(\Kk=\{1,\dots,K\}\).
Particles of type \(i\) move according to a symmetric
\(\alpha_i\)-stable process and reproduce according to a critical
offspring law with irreducible stochastic mean matrix.
All lifetime distributions are non-arithmetic. The type-\(1\) lifetime distribution has infinite
mean and a regularly varying tail
\[
  1-F_1(t)\sim c_1t^{-\gamma},
  \qquad 0<\gamma<1,
\]
whereas the lifetime distributions of types \(2,\dots,K\) satisfy the
polynomial upper-tail bounds
\[
  1-F_i(t)\le C t^{-\eta_i},
  \qquad i=2,\dots,K,
  \qquad \eta_i>1,
  \qquad
  \eta:=\min_{2\le i\le K}\eta_i.
\]
The offspring mechanism satisfies a \((1+\beta)\)-stable-domain
regular-variation condition, with \(\beta\in(0,1]\).
The initial population is a Poisson random measure with intensity
\[
  \Lambda=\sum_{i=1}^K a_i\,\lambda\otimes\delta_i,
  \qquad a_i\ge0,
\]
and we set \(A:=\sum_{i=1}^K a_i\).
Under the space--lifetime condition
\[
  \rho:=\left(\eta-1\right)\wedge\frac{N}{\alpha_1}
  >
  \frac{\gamma}{\beta},
\]
we prove that the system converges to a Poisson random measure with
intensity \(A\lambda\otimes\delta_1\).
Thus, although the initial population may assign positive spatial
intensity to every type, the limiting population is supported entirely
on the infinite-mean type while preserving the aggregate initial
spatial intensity \(A\lambda\).
\end{abstract}
\maketitle

\section{Introduction and previous work}

This paper studies the limiting behaviour of a critical branching particle system viewed as a measure-valued process. To formulate the relevant limiting regimes precisely, we use the following notation and terminology. Let \(E\) be a locally compact state space, and let \((X_t)_{t\ge0}\) be an \(\cM(E)\)-valued particle system, where \(\cM(E)\) denotes the space of nonnegative Radon measures on \(E\), endowed with the vague topology. We write \(C_c^+(E)\) for the class of nonnegative continuous functions on \(E\) with compact support. For \(\mu\in\cM(E)\) and a measurable function \(\varphi:E\to\R\) that is integrable with respect to \(\mu\), we write
\[
  \ip{\mu}{\varphi}:=\int_E\varphi\,d\mu.
\]
We say that the system survives in the limiting sense if
\[
  X_t\Rightarrow X_\infty,\qquad t\to\infty,
\]
for some random measure \(X_\infty\in\cM(E)\) satisfying
\[
  \Pbb\{X_\infty\ne0\}>0.
\]
We say that the system undergoes local extinction if
\[
  \ip{X_t}{\varphi}\xrightarrow[t\to\infty]{P}0,
  \qquad \varphi\in C_c^+(E).
\]
When the limiting random measure preserves the initial mean intensity, we say that the system persists with full intensity.

For monotype critical branching particle systems in \(\R^N\), a basic reference is the Markovian model studied in \cite{GW}. In that model, the initial population is a homogeneous Poisson random measure on \(\R^N\) with intensity proportional to Lebesgue measure. Each particle moves according to a symmetric \(\alpha\)-stable process, \(0<\alpha\le2\), has an exponentially distributed lifetime, and, when it dies, produces a random number of offspring whose probability generating function is
\[
  H(s)=s+\frac12(1-s)^{1+\beta},
  \qquad s\in[0,1],\quad 0<\beta\le1.
\]
All offspring particles start from the position at which their parent died and evolve according to the same rules. For each particle, its spatial motion, lifetime, and offspring variable are mutually independent, and the corresponding random elements are independent across particles. For this system, \cite{GW} proved local extinction when \(N\le\alpha/\beta\) and convergence to a nonzero equilibrium of full intensity when \(N>\alpha/\beta\).

The case of general non-arithmetic lifetime distributions was treated in \cite{VW}. In the finite-mean case, the same critical dimension \(N=\alpha/\beta\) is obtained: local extinction occurs when \(N\le\alpha/\beta\), whereas convergence to a nontrivial limit of full intensity holds when \(N>\alpha/\beta\). If the lifetime tail is regularly varying with index \(-\gamma\), where \(0<\gamma\le1\), the critical dimension becomes \(N=\alpha\gamma/\beta\). More precisely, local extinction occurs when \(N<\alpha\gamma/\beta\), convergence to a Poisson random measure of full intensity holds when \(N>\alpha\gamma/\beta\), and, when \(N=\alpha\gamma/\beta\), the system converges to a nontrivial limit along a subsequence. Thus, in the monotype setting, the lifetime-tail exponent \(\gamma\) changes the critical dimension.

We next turn to the multitype setting. Consider a branching particle system in \(\R^N\) with finite type space \(\Kk:=\{1,\dots,K\}\). A particle of type \(i\), born at \(x\in\R^N\), moves from \(x\) according to a symmetric \(\alpha_i\)-stable process, with \(0<\alpha_i\le2\), and has lifetime distribution \(F_i\). At the end of its lifetime, the particle dies and produces \(\zeta_{i,j}\) offspring particles of type \(j\), \(j=1,\dots,K\), where
\[
  \zeta_i:=(\zeta_{i,1},\dots,\zeta_{i,K})
\]
denotes a generic offspring vector for a type-\(i\) particle. All offspring particles start from the position at which their parent died. The multitype offspring probability generating function is
\[
  f_i(\mathbf s):=\E\left(\prod_{j=1}^K s_j^{\zeta_{i,j}}\right),
  \qquad
  \mathbf s=(s_1,\dots,s_K)\in[0,1]^K.
\]
Particles of the same type obey the same probabilistic laws. For each particle, its spatial motion, lifetime, and offspring vector are mutually independent, and the corresponding random elements are independent across particles. Write \(\mathbf f:=(f_1,\dots,f_K)\) and \(\one:=(1,\dots,1)\). The mean offspring matrix is
\[
  M:=(m_{i,j})_{i,j=1}^K,
  \qquad
  m_{i,j}:=\frac{\partial f_i}{\partial s_j}(\one)
  =\E(\zeta_{i,j}).
\]
Throughout the paper, \(M\) is assumed to be irreducible and stochastic. Consequently, its Perron--Frobenius eigenvalue is \(1\), and the system is critical. We set \(u:=\one\) and let \(v\) denote the stationary probability vector of \(M\); then
\[
  Mu=u,\qquad vM=v,\qquad \ip{v}{u}=1.
\]

The extinction side for multitype Bellman--Harris systems was studied in \cite{Kevei}. The system considered there starts at time \(0\) from a Poisson random measure on \(\R^N\times\Kk\) with prescribed homogeneous intensity
\[
  \Lambda_0:=\sum_{i=1}^K b_i\,\lambda\otimes\delta_i,
  \qquad b_i\ge0,
\]
where \(\lambda\) denotes Lebesgue measure on \(\R^N\). All particles present at time \(0\) have age \(0\), and the initial population is independent of all particle motions, lifetimes, and offspring variables. The offspring mechanism is assumed to satisfy
\[
  x-\ip{v}{\one-\mathbf f(\one-ux)}
  \sim x^{1+\beta}L(x),
  \qquad x\downarrow0,
\]
where \(L\) is a positive function slowly varying at zero and \(0<\beta\le1\). Let \(\mathbf f^{(n)}\) denote the \(n\)-fold iterate of \(\mathbf f\), defined recursively by
\[
  \mathbf f^{(1)}:=\mathbf f,
  \qquad
  \mathbf f^{(n+1)}:=\mathbf f\circ\mathbf f^{(n)},
  \qquad n\ge1,
\]
and set \(\mathbf 0:=(0,\dots,0)\). In the finite-mean case, the lifetime distributions are assumed to be non-arithmetic and to satisfy
\[
  \lim_{n\to\infty}
  \frac{n\bigl(1-F_i(n)\bigr)}
       {\ip{v}{\one-\mathbf f^{(n)}(\mathbf 0)}}
  =0,
  \qquad i=1,\dots,K.
\]
Under these assumptions, the system undergoes local extinction when
\[
  N<\frac{\alpha_*}{\beta},
  \qquad
  \alpha_*:=\min_{1\le i\le K}\alpha_i.
\]

The survival side of the finite-mean multitype model was treated in \cite{KLRFiniteMean}. In that result, the type weights of the initial intensity are not arbitrary. For \(i=1,\dots,K\), define
\[
  \mu_i:=\int_0^\infty\bigl(1-F_i(t)\bigr)\,dt<\infty,
  \qquad
  \overline\mu:=\sum_{j=1}^K v_j\mu_j,
  \qquad
  \theta_i:=\frac{v_i\mu_i}{\overline\mu}.
\]
Then \(\sum_{i=1}^K\theta_i=1\), and the initial population is a Poisson random measure on \(\R^N\times\Kk\) with intensity
\[
  \overline\Lambda
  :=\sum_{i=1}^K\theta_i\,\lambda\otimes\delta_i.
\]
The lifetime distributions are non-arithmetic and satisfy, for some \(C>0\),
\[
  1-F_i(t)\le Ct^{-\eta_i},
  \qquad t>0,\quad i=1,\dots,K,
  \qquad \eta_i>1.
\]
Set
\[
  \eta:=\min_{1\le i\le K}\eta_i.
\]
The authors prove that, if
\[
  (\eta-1)\wedge\frac{N}{\alpha_*}>\frac1\beta,
\]
then
\[
  X_t\Rightarrow X_\infty,\qquad t\to\infty,
\]
where \(X_\infty\) is a nonzero random measure that preserves the initial intensity:
\[
  \E\ip{X_\infty}{\varphi}
  =\ip{\overline\Lambda}{\varphi},
  \qquad
  \varphi\in C_c^+(\R^N\times\Kk).
\]
Consequently, the system persists with full intensity.

The preceding results leave open the survival problem in the mixed finite--infinite mean regime. In the present paper, we address this problem by proving convergence of the particle system under the assumptions stated below. Specifically, type \(1\) has an infinite-mean lifetime, whereas types \(2,\dots,K\) have finite-mean lifetimes. Write \(\overline F_i(t):=1-F_i(t)\). We assume that \(F_i\) is non-arithmetic for every \(i\in\Kk\), and that
\[
  \overline F_1(t)\sim c_1t^{-\gamma},
  \qquad t\to\infty,
  \qquad c_1>0,\quad 0<\gamma<1,
\]
whereas
\[
  \overline F_i(t)\le Ct^{-\eta_i},
  \qquad t>0,\quad i=2,\dots,K,
  \qquad \eta_i>1.
\]
Set
\[
  \eta:=\min_{2\le i\le K}\eta_i,
\]
and assume that \(F_i(0)=0\) and \(m_{i,i}>0\) for every \(i\in\Kk\).

Let \(S:=\R^N\times\Kk\). The initial population is a Poisson random measure on \(S\) with intensity
\[
  \Lambda:=\sum_{i=1}^K a_i\,\lambda\otimes\delta_i,
  \qquad a_i\ge0,
\]
and we set
\[
  A:=\sum_{i=1}^K a_i.
\]
If
\[
  \rho:=\left(\eta-1\right)\wedge\frac{N}{\alpha_1}
  >
  \frac{\gamma}{\beta},
\]
then
\[
  X_t\Rightarrow X_\infty,
  \qquad t\to\infty,
\]
where \(X_\infty\) is a Poisson random measure on \(S\) with intensity
\[
  A\lambda\otimes\delta_1.
\]
Hence the total spatial intensity \(A\lambda\) is preserved, while the limiting measure is supported on the infinite-mean type. Moreover, if \(A>0\), then
\[
  \Pbb\{X_\infty\ne0\}=1.
\]
This result complements the mixed-regime extinction results in \cite{Kevei}.

The proof reduces the matrix renewal system to the scalar renewal process generated by successive return cycles of the embedded type chain to type \(1\). The return-cycle length is tail-equivalent to the type-\(1\) lifetime, while the finite-mean part of the cycle has tail \(o(t^{-\gamma})\). The convolution estimates required for the nonlinear remainder then follow from the elementary renewal identity and Blackwell's fixed-window renewal theorem for infinite-mean renewal processes; see \cite[Eq.~(1.1), p.~88]{Thorisson1987}.

The paper is organised as follows. Section~\ref{sec:model} defines the mixed multitype branching particle system and states the main theorem. Section~\ref{sec:main-proof} proves the theorem. Section~\ref{sec:auxiliary} contains the local hitting estimate, the Stieltjes convolution estimates, and the mixed finite--infinite mean renewal theorem. Section~\ref{sec:examples} gives examples satisfying the lifetime assumptions.

\section{The model and main results}\label{sec:model}

\subsection{The model}

The population at time \(t\ge0\) is represented by
\[
  X_t=\sum_k\delta_{(X_t^k,K_t^k)}\in\cM(S),
\]
where \(X_t^k\in\R^N\) and \(K_t^k\in\Kk\) are, respectively, the position and type of the \(k\)-th particle alive at time \(t\). For \((x,i)\in S\), let \(\Pbb_{(x,i)}\) and \(\E_{(x,i)}\) denote probability and expectation, respectively, for the system started from a single particle at \((x,i)\).

\subsubsection{Spatial motion}

A particle of type \(i\in\Kk\), born at \(x\in\R^N\), has position \(x+\xi_t^{(i)}\) at age \(t\), where \((\xi_t^{(i)})_{t\ge0}\) is an isotropic symmetric \(\alpha_i\)-stable process in \(\R^N\) started from \(0\), with \(\alpha_i\in(0,2]\). For every bounded Borel function \(g:\R^N\to\R\), define
\[
  P_{t,\alpha_i}g(x):=\E\bigl[g(x+\xi_t^{(i)})\bigr].
\]
For every \(g\in L^1(\R^N)\),
\[
  \int_{\R^N}P_{t,\alpha_i}g(x)\,dx
  =
  \int_{\R^N}g(x)\,dx,
  \qquad t\ge0.
\]

\subsubsection{Lifetimes}

The lifetime of a type-\(i\) particle has distribution function \(F_i\) on \([0,\infty)\), with \(F_i(0)=0\). Each \(F_i\) is non-arithmetic, that is, it is not supported on \(h\N_0\) for any \(h>0\), where \(\N_0:=\{0,1,2,\dots\}\). The distribution \(F_1\) satisfies
\begin{equation}\label{eq:F1-heavy-tail}
  \overline F_1(t)\sim c_1t^{-\gamma},
  \qquad t\to\infty,
  \qquad c_1>0,\quad 0<\gamma<1.
\end{equation}
For some \(C>0\) and exponents \(\eta_i>1\), the remaining lifetime distributions satisfy
\begin{equation}\label{eq:finite-type-tails}
  \overline F_i(t)\le Ct^{-\eta_i},
  \qquad t>0,\quad i=2,\dots,K.
\end{equation}
Set
\begin{equation}\label{eq:eta-def}
  \eta:=\min_{2\le i\le K}\eta_i>1
\end{equation}
and
\begin{equation}\label{eq:rho-def}
  \rho:=\left(\eta-1\right)\wedge\frac{N}{\alpha_1}.
\end{equation}

\subsubsection{Branching mechanism}

At its death time, a type-\(i\) particle produces \(\zeta_{i,j}\) offspring particles of type \(j\), \(j=1,\dots,K\), all born at the position at which their parent dies. The corresponding offspring vector is
\[
  \zeta_i:=(\zeta_{i,1},\dots,\zeta_{i,K})\in\N_0^K,
\]
and its probability generating function is
\[
  f_i(\mathbf s):=\E\left(\prod_{j=1}^K s_j^{\zeta_{i,j}}\right),
  \qquad
  \mathbf s=(s_1,\dots,s_K)\in[0,1]^K.
\]
For each particle, its spatial motion, lifetime, and offspring vector are mutually independent, and the corresponding triples are independent across particles. Conditional on their types and birth positions, all offspring particles evolve independently according to the type-dependent laws specified above.

Let \(\mathbf f:=(f_1,\dots,f_K)\) and \(\one:=(1,\dots,1)\). The mean offspring matrix is
\[
  M:=(m_{i,j})_{i,j=1}^K,
  \qquad
  m_{i,j}:=\frac{\partial f_i}{\partial s_j}(\one)
  =\E(\zeta_{i,j}).
\]
We assume that \(M\) is irreducible and stochastic and that
\begin{equation}\label{eq:positive-diagonal}
  m_{i,i}>0,
  \qquad i=1,\dots,K.
\end{equation}
Set \(u:=\one\), and let \(v=(v_1,\dots,v_K)\) be the stationary distribution of \(M\). Then
\[
  Mu=u,\qquad vM=v,\qquad
  \sum_{i=1}^K v_i=1,\qquad v_i>0,\qquad
  \ip{v}{u}=1.
\]
In particular, the Perron--Frobenius eigenvalue of \(M\) is \(1\), and the system is critical. We assume that, for some \(\beta\in(0,1]\) and some positive function \(L\) slowly varying at zero,
\begin{equation}\label{eq:branching-condition}
  x-\ip{v}{\one-\mathbf f(\one-ux)}
  \sim x^{1+\beta}L(x),
  \qquad x\downarrow0.
\end{equation}

\subsubsection{Initial configuration}

The initial configuration \(X_0\) is a Poisson random measure on \(S\) with intensity
\begin{equation}\label{eq:initial-intensity}
  \Lambda:=\sum_{i=1}^K a_i\,\lambda\otimes\delta_i,
  \qquad a_i\ge0.
\end{equation}
It is independent of all particle motions, lifetimes, and offspring vectors. Set
\begin{equation}\label{eq:A-total}
  A:=\sum_{i=1}^K a_i.
\end{equation}

\subsubsection{Associated Markov--renewal process}
\label{sec:mixed-renewal-framework}

Define the matrix of renewal kernels \(F=(F_{i,j})_{i,j=1}^K\) by
\[
  F_{i,j}(t):=m_{i,j}F_i(t),
  \qquad t\ge0,\quad i,j=1,\dots,K.
\]
Let \((Y_n)_{n\ge0}\) be the Markov chain on \(\Kk\) with transition matrix \(M\). Conditional on \((Y_n)_{n\ge0}\), let \((\tau_n)_{n\ge0}\) be independent random variables such that \(\tau_n\) has distribution function \(F_{Y_n}\), and define
\[
  S_0:=0,\qquad
  S_n:=\sum_{k=0}^{n-1}\tau_k,
  \qquad n\ge1.
\]
For \(i\in\Kk\), let \(\Pbb_i\) and \(\E_i\) denote probability and expectation, respectively, for this Markov--renewal process started from \(Y_0=i\). Set
\[
  \sigma_0:=0,\qquad
  \sigma_{n+1}:=\inf\{m>\sigma_n:Y_m=1\},
  \qquad n\ge0.
\]
Thus, \(\sigma_1\) is the first return time to type \(1\) under \(\Pbb_1\) and the first hitting time of type \(1\) under \(\Pbb_i\) when \(i\ne1\). Under \(\Pbb_1\), define the first return-cycle length by
\[
  C:=S_{\sigma_1}.
\]
Let \(F_C\) denote the law of \(C\) under \(\Pbb_1\), and define its renewal measure by
\[
  U_C:=\delta_0+\sum_{n=1}^{\infty}F_C^{*n}.
\]
We also write
\[
  U_C(t):=U_C([0,t]),
  \qquad t\ge0.
\]

\subsubsection{Laplace functionals}

For \(\varphi\in C_c^+(S)\), define
\[
  W_t(x,i):=\E_{(x,i)}\left(e^{-\ip{X_t}{\varphi}}\right),
  \qquad
  V_t(x,i):=1-W_t(x,i),
  \qquad (x,i)\in S.
\]
Set
\begin{equation}\label{eq:Gi-def}
  G_i(t):=\|V_t(\cdot,i)\|_1
  =\int_{\R^N}V_t(x,i)\,dx,
  \qquad i=1,\dots,K.
\end{equation}
For the Poisson initial configuration \eqref{eq:initial-intensity}, proceeding as in \cite[Proposition~2.1(b)]{KLRFiniteMean}, we obtain
\begin{equation}\label{eq:poisson-laplace}
\begin{split}
  \E\left(e^{-\ip{X_t}{\varphi}}\right)
  &=
  \exp\left\{
    -\sum_{i=1}^K a_i
    \int_{\R^N}
    \E_{(x,i)}\left(1-e^{-\ip{X_t}{\varphi}}\right)dx
  \right\}\\
  &=
  \exp\left\{-\sum_{i=1}^K a_iG_i(t)\right\}.
\end{split}
\end{equation}
\subsection{Main results}\label{subsec:main-theorem}

\begin{theorem}\label{thm:main-mixed}
Consider the critical multitype branching particle system defined in Section~\ref{sec:model}. Assume the space--lifetime condition
\begin{equation}\label{eq:main-dim-condition}
  \rho>\frac{\gamma}{\beta}.
\end{equation}
Then, for every \(\varphi\in C_c^+(S)\),
\begin{equation}\label{eq:Gi-main-limit}
  \lim_{t\to\infty}G_i(t)
  =
  G_1(0)
  =
  \int_{\R^N}\left(1-e^{-\varphi(x,1)}\right)dx,
  \qquad i=1,\dots,K.
\end{equation}
Consequently, \eqref{eq:poisson-laplace} yields
\begin{equation}\label{eq:laplace-main-limit}
  \E\left(e^{-\ip{X_t}{\varphi}}\right)
  \longrightarrow
  \exp\left\{
    -A\int_{\R^N}\left(1-e^{-\varphi(x,1)}\right)dx
  \right\},
  \qquad t\to\infty.
\end{equation}
Therefore,
\[
  X_t\Rightarrow X_\infty,
  \qquad t\to\infty,
\]
where \(X_\infty\) is a Poisson random measure on \(S\) with intensity
\[
  A\lambda\otimes\delta_1.
\]
If \(A>0\), then
\[
  \Pbb\{X_\infty\ne0\}=1.
\]
\end{theorem}

\begin{remark}
Assume \(A>0\), and set
\[
  \alpha:=\min_{1\le i\le K}\alpha_i,
  \qquad
  N_c:=\frac{\alpha_1\gamma}{\beta}.
\]
Condition \eqref{eq:main-dim-condition} implies \(N>N_c\). We compare Theorem~\ref{thm:main-mixed} with the local-extinction results in \cite[Theorems~2--4]{Kevei}.

\emph{Case A: \(\alpha=\alpha_1\).} If
\[
  \eta-1>\frac{N}{\alpha},
\]
then local extinction holds for \(N<N_c\). Hence, whenever the hypotheses of both results are satisfied, local extinction holds below \(N_c\), survival holds above \(N_c\), and the case \(N=N_c\) remains open.

\emph{Case B: \(\alpha<\alpha_1\).} Suppose first that
\[
  \alpha\ge\alpha_1\gamma
  \qquad\text{and}\qquad
  \gamma\eta>\frac{N}{\alpha}+1.
\]
Then local extinction holds for \(N<N_c\). Thus, under the hypotheses of both results, the only dimension not covered by the extinction and survival conclusions is \(N=N_c\).

Suppose now that
\[
  \alpha<\alpha_1\gamma
  \qquad\text{and}\qquad
  \gamma\eta>\frac{N}{\alpha}+1.
\]
In this case, local extinction holds for
\[
  N<N_+,
  \qquad
  N_+:=
  \frac{\gamma\alpha\alpha_1}
       {\alpha_1(\beta+1)\gamma-\alpha}.
\]
Since \(N_+<N_c\), the available extinction and survival results leave the interval
\[
  N_+\le N\le N_c
\]
unresolved.
\end{remark}

\section{Proof of the main results}\label{sec:main-proof}

\begin{proof}[Proof of Theorem~\ref{thm:main-mixed}]
Let \(\psi\in C_c^+(S)\), set \(\psi_i(x):=\psi(x,i)\), and define
\[
  H_i^\psi(t):=\int_{\R^N}\E_{(x,i)}\bigl(\ip{X_t}{\psi}\bigr)\,dx,
  \qquad i=1,\dots,K.
\]
By \cite[Proposition~2.1(a)]{KLRFiniteMean}, \(H^\psi:=(H_1^\psi,\dots,H_K^\psi)\) is locally bounded and satisfies
\begin{equation}\label{eq:first-moment-mixed}
  H_i^\psi(t)
  =
  \overline F_i(t)\|\psi_i\|_1
  +\sum_{j=1}^K m_{i,j}\int_{[0,t]}H_j^\psi(t-s)\,dF_i(s),
  \qquad i=1,\dots,K.
\end{equation}
Thus \(H^\psi\) solves the linear renewal system with
\[
  z_i^\psi(t):=\overline F_i(t)\|\psi_i\|_1,
  \qquad i=1,\dots,K.
\]
For \(i=2,\dots,K\), \eqref{eq:F1-heavy-tail}, \eqref{eq:finite-type-tails}, and \(\eta_i>1>\gamma\) give
\[
  z_i^\psi(t)=O(t^{-\eta_i})=o\bigl(\overline F_1(t)\bigr),
  \qquad t\to\infty.
\]
Moreover,
\[
  z_1^\psi(t)=a\overline F_1(t)+r_1(t),
  \qquad
  a=\|\psi_1\|_1,
  \qquad
  r_1\equiv0.
\]
Thus \(r_1,z_2^\psi,\dots,z_K^\psi\) satisfy condition \textup{(H)} of Theorem~\ref{thm:mixed-renewal}, which yields
\begin{equation}\label{eq:first-moment-limit-mixed}
  H_i^\psi(t)\longrightarrow\|\psi_1\|_1,
  \qquad t\to\infty,\quad i=1,\dots,K.
\end{equation}
Together with the local boundedness of \(H^\psi\), this implies
\begin{equation}\label{eq:first-moment-bound-mixed}
  \sup_{t\ge0}\sum_{i=1}^K H_i^\psi(t)<\infty.
\end{equation}
Taking \(\psi=\varphi\) and using
\[
  V_t(x,i)
  =
  \E_{(x,i)}\left(1-e^{-\ip{X_t}{\varphi}}\right)
  \le
  \E_{(x,i)}\bigl(\ip{X_t}{\varphi}\bigr),
\]
we obtain
\begin{equation}\label{eq:G-L1-bound}
  \sup_{t\ge0}\sum_{i=1}^K G_i(t)<\infty.
\end{equation}

Let \(D_\varphi\subset\R^N\) be compact and satisfy
\[
  \operatorname{supp}\varphi\subset D_\varphi\times\Kk.
\]
Since
\[
  1-e^{-\ip{X_t}{\varphi}}
  \le \ind{X_t(D_\varphi\times\Kk)>0},
\]
Lemma~\ref{lem:local-hitting-kevei} gives
\begin{equation}\label{eq:V-sup-bound-mixed}
  \sup_{x\in\R^N,\ i\in\Kk}V_t(x,i)
  \le C_\varphi(1+t)^{-\rho},
  \qquad t\ge0.
\end{equation}

For \(y=(y_1,\dots,y_K)\in[0,1]^K\), define
\begin{equation}\label{eq:R-def-main}
  R_i(y):=\sum_{j=1}^K m_{i,j}y_j-\bigl(1-f_i(\one-y)\bigr),
  \qquad i=1,\dots,K,
\end{equation}
and
\begin{equation}\label{eq:B-def-main}
  B_i(t):=\int_{\R^N}R_i\bigl(V_t(x,1),\dots,V_t(x,K)\bigr)\,dx.
\end{equation}
The proof of \cite[Theorem~2.3]{KLRFiniteMean} shows that each \(R_i\) is nonnegative and coordinatewise nondecreasing on \([0,1]^K\), and that each \(B_i\) is nonnegative, Borel measurable, and locally bounded. Moreover, for every \(\delta\in(0,\beta)\), after increasing the constant if necessary,
\begin{equation}\label{eq:Ri-bound}
  R_i(y)\le C_\delta\|y\|_\infty^{1+\beta-\delta},
  \qquad y\in[0,1]^K,\quad i=1,\dots,K.
\end{equation}
The same proof gives
\begin{equation}\label{eq:G-renewal-nonlinear}
  G_i(t)
  =
  z_i(t)+\sum_{j=1}^K m_{i,j}\int_{[0,t]}G_j(t-s)\,dF_i(s),
  \qquad i=1,\dots,K,
\end{equation}
where
\begin{equation}\label{eq:z-main-def}
  z_i(t):=\overline F_i(t)G_i(0)-\int_{[0,t]}B_i(t-s)\,dF_i(s).
\end{equation}
In particular, each \(z_i\) is Borel measurable and locally bounded.

By \eqref{eq:Ri-bound}, \eqref{eq:V-sup-bound-mixed}, and \eqref{eq:G-L1-bound},
\[
\begin{split}
  B_i(t)
  &\le C\int_{\R^N}
  \left(\max_{1\le j\le K}V_t(x,j)\right)^{1+\beta-\delta}dx\\
  &\le
  C\left(\sup_{x\in\R^N,\ j\in\Kk}V_t(x,j)\right)^{\beta-\delta}
  \sum_{j=1}^K\|V_t(\cdot,j)\|_1\\
  &\le C(1+t)^{-\rho(\beta-\delta)}.
\end{split}
\]
By the space--lifetime condition \eqref{eq:main-dim-condition}, choose \(\delta\in(0,\beta)\) sufficiently small that
\begin{equation}\label{eq:kappa-choice}
  \kappa:=\rho(\beta-\delta)>\gamma.
\end{equation}
Consequently,
\begin{equation}\label{eq:B-kappa-bound}
  B_i(t)\le C(1+t)^{-\kappa},
  \qquad t\ge0,\quad i=1,\dots,K.
\end{equation}

We now verify the assumptions of Theorem~\ref{thm:mixed-renewal}. For \(i=1\), write
\[
  z_1(t)=G_1(0)\overline F_1(t)+r_1(t),
  \qquad
  r_1(t):=-\int_{[0,t]}B_1(t-s)\,dF_1(s).
\]
By \eqref{eq:B-kappa-bound} and Lemma~\ref{lem:F1-stieltjes},
\[
  |r_1(t)|=o\bigl(\overline F_1(t)\bigr),
  \qquad t\to\infty.
\]
Thus \(r_1\) satisfies condition \textup{(H)} of Theorem~\ref{thm:mixed-renewal}.

For \(i=2,\dots,K\),
\[
  \overline F_i(t)G_i(0)\le C(1+t)^{-\eta_i},
\]
where \(\eta_i>1>\gamma\). By \eqref{eq:B-kappa-bound} and Lemma~\ref{lem:finite-tail-stieltjes},
\[
  \int_{[0,t]}B_i(t-s)\,dF_i(s)
  \le C(1+t)^{-\kappa_i'},
  \qquad
  \kappa_i':=\min\{\eta_i,\kappa\}>\gamma.
\]
Therefore,
\[
  |z_i(t)|\le C(1+t)^{-\kappa_i'},
  \qquad i=2,\dots,K.
\]
Since \(\overline F_1(t)\sim c_1t^{-\gamma}\) and \(\kappa_i'>\gamma\),
\[
  \frac{|z_i(t)|}{\overline F_1(t)}
  =O\bigl(t^{\gamma-\kappa_i'}\bigr)
  \longrightarrow0,
  \qquad t\to\infty.
\]
Hence each \(z_i\), \(i\ge2\), satisfies condition \textup{(H)} of Theorem~\ref{thm:mixed-renewal}.

Applying Theorem~\ref{thm:mixed-renewal} to \eqref{eq:G-renewal-nonlinear} with \(a=G_1(0)\), we obtain
\[
  G_i(t)\longrightarrow G_1(0),
  \qquad t\to\infty,\quad i=1,\dots,K.
\]
This proves \eqref{eq:Gi-main-limit}.

By \eqref{eq:poisson-laplace},
\[
  \E\left(e^{-\ip{X_t}{\varphi}}\right)
  =
  \exp\left\{-\sum_{i=1}^K a_iG_i(t)\right\}.
\]
Using \eqref{eq:Gi-main-limit}, we obtain
\[
  \sum_{i=1}^K a_iG_i(t)
  \longrightarrow
  \left(\sum_{i=1}^K a_i\right)G_1(0)
  =
  A\int_{\R^N}\left(1-e^{-\varphi(x,1)}\right)dx.
\]
This proves \eqref{eq:laplace-main-limit}. The right-hand side is the Laplace functional of a Poisson random measure with intensity \(A\lambda\otimes\delta_1\).

It remains only to justify tightness. As in the proof of
Proposition~2.1(a) of \cite{KLRFiniteMean}, let \(K_0\subset S\) be compact
and choose \(\psi\in C_c^+(S)\) such that \(\psi\ge\mathbf1_{K_0}\).
By Markov's inequality, the Poisson initial intensity, and
\eqref{eq:first-moment-bound-mixed},
\[
  \Pbb\{X_t(K_0)>b\}
  \le \frac{\E\ip{X_t}{\psi}}{b}
  =\frac{1}{b}\sum_{i=1}^Ka_iH_i^\psi(t)
  \le\frac{C_{K_0}}{b},
  \qquad t\ge0.
\]
Hence
\[
  \lim_{b\to\infty}\sup_{t\ge0}\Pbb\{X_t(K_0)>b\}=0.
\]
By the tightness criterion for random Radon measures,
\cite[Lemma~4.5]{Olav}, the family
\(\{\mathcal L(X_t):t\ge0\}\) is tight in the vague topology.

Recall that \(S=\R^N\times\Kk\), where \(\Kk\) is finite. Hence \(S\) is a
locally compact, second countable Hausdorff space, and the convergence
criterion for random measures in the vague topology,
\cite[Theorem~4.2]{Olav}, applies on \(S\).

By \eqref{eq:laplace-main-limit}, for every \(\varphi\in C_c^+(S)\),
\[
  \E\exp\{-\ip{X_t}{\varphi}\}
  \longrightarrow
  \E\exp\{-\ip{X_\infty}{\varphi}\}.
\]
Thus the Laplace functionals converge on the class
\(C_c^+(S)\). By Theorem~4.2 of \cite{Olav},
\[
  X_t\Rightarrow X_\infty,\qquad t\to\infty,
\]
in the space of Radon measures on \(S\) endowed with the vague topology. If
\(A>0\), the limiting Poisson random measure has infinite total mass on
\(\R^N\times\{1\}\), and in particular is nonzero almost surely. The proof is
complete.

\end{proof}

\section{\texorpdfstring{Auxiliary results}{Auxiliary results}}\label{sec:auxiliary}

Section~\ref{sec:auxiliary} contains the auxiliary renewal results, including
the local hitting estimate, the Stieltjes convolution estimates, and the mixed
finite--infinite mean renewal theorem.

\begin{lemma}
\label{lem:local-hitting-kevei}
Consider the branching particle system described in Section~\ref{sec:model}, and assume that
the lifetime distributions satisfy \eqref{eq:F1-heavy-tail} and
\eqref{eq:finite-type-tails}. Then, for every compact set
\(D\subset\mathbb R^N\), there exists \(C_D>0\) such that

\begin{equation}\label{eq:hitting-bound}
  \sup_{x\in\R^N,\ i\in\Kk}
  \Pbb_{(x,i)}\{X_t(D\times\Kk)>0\}
  \le C_D(1+t)^{-\rho},
  \qquad t\ge0 .
\end{equation}
\end{lemma}
\begin{proof}
The argument in \cite[Lemma~11 and Section~4.3]{Kevei}, applied with initial age \(0\), remains valid under the present assumptions. Its input concerning the embedded type chain is ergodicity, which follows from the irreducibility of \(M\) and \eqref{eq:positive-diagonal}. Moreover, replacing the normalization \(\overline F_1(t)\sim t^{-\gamma}\) used there by \eqref{eq:F1-heavy-tail} only changes the constant in the estimate. Therefore,
\[
  \sup_{x\in\R^N,\ i\in\Kk}
  \Pbb_{(x,i)}\{X_t(D\times\Kk)>0\}
  \le C_D\left(t^{-N/\alpha_1}+t^{1-\eta}\right),
  \qquad t\ge1.
\]
Since
\[
  t^{-N/\alpha_1}+t^{1-\eta}
  \le C(1+t)^{-\rho},
  \qquad
  \rho=(\eta-1)\wedge\frac{N}{\alpha_1},
\]
the result follows after increasing \(C_D\) to cover \(0\le t<1\).
\end{proof}

\begin{lemma}\label{lem:infinite-mean-renewal-negligibility}
Let \(X\) be a strictly positive, non-arithmetic random variable with distribution function \(F\), renewal measure
\[
  U:=\sum_{n=0}^{\infty}F^{*n},
  \qquad F^{*0}:=\delta_0,
\]
and tail
\[
  \overline F(t):=1-F(t)=\Pbb\{X>t\}
  \sim ct^{-\gamma},
  \qquad t\to\infty,
  \qquad c>0,\quad 0<\gamma<1.
\]
Let \(h:[0,\infty)\to\R\) be locally bounded and Borel measurable, and extend it by zero to \((-\infty,0)\). If
\begin{equation}\label{eq:renewal-small-relative-tail}
  |h(t)|=o\bigl(\overline F(t)\bigr),
  \qquad t\to\infty,
\end{equation}
then
\begin{equation}\label{eq:renewal-negligibility-X}
  \int_{[0,t]}h(t-s)\,U(ds)\longrightarrow0,
  \qquad t\to\infty.
\end{equation}
In particular, for every \(\kappa>\gamma\),
\begin{equation}\label{eq:renewal-power-negligibility-X}
  \int_{[0,t]}(1+t-s)^{-\kappa}\,U(ds)\longrightarrow0,
  \qquad t\to\infty.
\end{equation}
\end{lemma}

\begin{proof}
Write \(U(t):=U([0,t])\). Let \((X_n)_{n\ge1}\) be independent copies of \(X\), and set
\[
  Z_0:=0,
  \qquad
  Z_n:=X_1+\cdots+X_n,
  \qquad n\ge1.
\]
The elementary renewal identity gives
\begin{equation}\label{eq:renewal-tail-identity-general}
  \int_{[0,t]}\overline F(t-s)\,U(ds)=1,
  \qquad t\ge0.
\end{equation}
Indeed, since \(X>0\) almost surely, \(Z_n\to\infty\) almost surely, and therefore
\[
\begin{split}
  \int_{[0,t]}\overline F(t-s)\,U(ds)
  &=\sum_{n=0}^{\infty}
  \E\left[\overline F(t-Z_n)\ind{Z_n\le t}\right]\\
  &=\sum_{n=0}^{\infty}\Pbb\{Z_n\le t<Z_{n+1}\}
  =1.
\end{split}
\]

Fix \(\varepsilon>0\). By \eqref{eq:renewal-small-relative-tail}, there exists \(T>0\) such that
\[
  |h(u)|\le\varepsilon\overline F(u),
  \qquad u\ge T.
\]
For \(t>T\),
\[
\begin{split}
  \left|\int_{[0,t]}h(t-s)\,U(ds)\right|
  &\le
  \int_{[0,t-T]}|h(t-s)|\,U(ds)
  +\int_{(t-T,t]}|h(t-s)|\,U(ds).
\end{split}
\]
By \eqref{eq:renewal-tail-identity-general}, the first term is bounded by \(\varepsilon\). Since \(h\) is locally bounded,
\[
  M_T:=\sup_{0\le u\le T}|h(u)|<\infty,
\]
and the second term is bounded by
\[
  M_T\bigl(U(t)-U(t-T)\bigr).
\]
The tail assumption and \(0<\gamma<1\) imply
\[
  \E X=\int_0^\infty\overline F(u)\,du=\infty.
\]
Hence Blackwell's fixed-window renewal theorem for non-arithmetic distributions, including the infinite-mean case, gives
\[
  U(t)-U(t-T)\longrightarrow0,
  \qquad t\to\infty;
\]
see \cite[Eq.~(1.1), p.~88]{Thorisson1987}. Consequently,
\[
  \limsup_{t\to\infty}
  \left|\int_{[0,t]}h(t-s)\,U(ds)\right|
  \le\varepsilon.
\]
Letting \(\varepsilon\downarrow0\) proves \eqref{eq:renewal-negligibility-X}.

Finally, if \(\kappa>\gamma\), then
\[
  \frac{(1+t)^{-\kappa}}{\overline F(t)}
  \sim c^{-1}t^{\gamma-\kappa}
  \longrightarrow0.
\]
Thus \(h(t):=(1+t)^{-\kappa}\) satisfies \eqref{eq:renewal-small-relative-tail}, and \eqref{eq:renewal-power-negligibility-X} follows from \eqref{eq:renewal-negligibility-X}.
\end{proof}

The next lemma is the precise estimate needed in the proof of the main theorem.

\begin{lemma}\label{lem:F1-stieltjes}
Assume \eqref{eq:F1-heavy-tail}. Let \(\kappa>\gamma\), and let \(h:[0,\infty)\to[0,\infty)\) be a Borel function such that, for some \(C_h>0\),
\begin{equation}\label{eq:h-kappa-bound}
  h(t)\le C_h(1+t)^{-\kappa},
  \qquad t\ge0.
\end{equation}
Then
\begin{equation}\label{eq:F1-conv-estimate}
  \int_{[0,t]}h(t-s)\,dF_1(s)
  =o\bigl(\overline F_1(t)\bigr),
  \qquad t\to\infty.
\end{equation}
\end{lemma}

\begin{proof}
Choose \(\theta\) such that
\[
  \frac{\gamma}{\kappa}<\theta<1,
\]
and set \(a_t:=t^\theta\). For all sufficiently large \(t\), write
\[
  I(t):=\int_{[0,t]}h(t-s)\,dF_1(s)=I_1(t)+I_2(t),
\]
where
\[
  I_1(t):=\int_{[0,t-a_t]}h(t-s)\,dF_1(s),
  \qquad
  I_2(t):=\int_{(t-a_t,t]}h(t-s)\,dF_1(s).
\]
On \([0,t-a_t]\), one has \(t-s\ge a_t\), and therefore
\[
  I_1(t)
  \le C_h(1+a_t)^{-\kappa}
  =O(t^{-\theta\kappa})
  =o(t^{-\gamma})
  =o\bigl(\overline F_1(t)\bigr).
\]
Moreover, \eqref{eq:h-kappa-bound} implies \(h\le C_h\), so
\[
\begin{split}
  I_2(t)
  &\le C_h\bigl(F_1(t)-F_1(t-a_t)\bigr)\\
  &=C_h\bigl(\overline F_1(t-a_t)-\overline F_1(t)\bigr).
\end{split}
\]
Since \(a_t/t=t^{\theta-1}\to0\), \eqref{eq:F1-heavy-tail} gives
\[
\begin{split}
  \frac{\overline F_1(t-a_t)}{\overline F_1(t)}
  &=
  \frac{\overline F_1(t-a_t)}{c_1(t-a_t)^{-\gamma}}
  \left(\frac{t}{t-a_t}\right)^\gamma
  \frac{c_1t^{-\gamma}}{\overline F_1(t)}
  \longrightarrow1.
\end{split}
\]
Hence
\[
  \overline F_1(t-a_t)-\overline F_1(t)
  =o\bigl(\overline F_1(t)\bigr),
\]
and consequently
\[
  I_2(t)=o\bigl(\overline F_1(t)\bigr).
\]
Combining the estimates for \(I_1(t)\) and \(I_2(t)\) proves \eqref{eq:F1-conv-estimate}.
\end{proof}

We also need a simpler convolution estimate for the finite-mean types.

\begin{lemma}\label{lem:finite-tail-stieltjes}
Let \(F\) be a distribution function on \([0,\infty)\) such that, for some \(C_F>0\) and \(\eta>\gamma\),
\[
  1-F(t)\le C_F(1+t)^{-\eta},
  \qquad t\ge0.
\]
Let \(h:[0,\infty)\to[0,\infty)\) be a Borel function such that, for some \(C_h>0\) and \(\kappa>\gamma\),
\[
  h(t)\le C_h(1+t)^{-\kappa},
  \qquad t\ge0.
\]
Set
\[
  \kappa':=\min\{\kappa,\eta\}.
\]
Then \(\kappa'>\gamma\), and there exists \(C'>0\) such that
\begin{equation}\label{eq:finite-tail-conv}
  \int_{[0,t]}h(t-s)\,dF(s)
  \le C'(1+t)^{-\kappa'},
  \qquad t\ge0.
\end{equation}
\end{lemma}

\begin{proof}
Let \(T\) be a random variable with distribution function \(F\). For \(t\ge0\), write
\[
\begin{split}
  \int_{[0,t]}h(t-s)\,dF(s)
  &=
  \E\left[h(t-T)\ind{T\le t}\right]=
  \E\left[h(t-T)\ind{T\le t/2}\right]
  +
  \E\left[h(t-T)\ind{t/2<T\le t}\right].
\end{split}
\]
On the event \(\{T\le t/2\}\),
\[
  1+t-T\ge\frac{1+t}{2},
\]
and therefore
\[
  \E\left[h(t-T)\ind{T\le t/2}\right]
  \le C_h2^\kappa(1+t)^{-\kappa}.
\]
Moreover, \(h\le C_h\), and hence
\[
\begin{split}
  \E&\left[h(t-T)\ind{t/2<T\le t}\right]
  \le C_h\Pbb\{T>t/2\}=C_h\bigl(1-F(t/2)\bigr)\le C_hC_F(1+t/2)^{-\eta}\le C_hC_F2^\eta(1+t)^{-\eta}.
\end{split}
\]
Combining these estimates gives
\[
\begin{split}
  \int_{[0,t]}h(t-s)\,dF(s)
  &\le
  C_h2^\kappa(1+t)^{-\kappa}
  +C_hC_F2^\eta(1+t)^{-\eta}\le
  C_h\bigl(2^\kappa+C_F2^\eta\bigr)
  (1+t)^{-\kappa'}.
\end{split}
\]
Thus \eqref{eq:finite-tail-conv} holds with
\[
  C':=C_h\bigl(2^\kappa+C_F2^\eta\bigr).
\]
Since \(\kappa>\gamma\) and \(\eta>\gamma\), one has \(\kappa'>\gamma\).
\end{proof}

\subsection{A mixed finite--infinite mean renewal theorem}
\label{sec:renewal}

The following theorem is the renewal result used in the proof of Theorem~\ref{thm:main-mixed}. We use the Markov--renewal notation introduced in Section~\ref{sec:mixed-renewal-framework}. Its proof uses Blackwell's fixed-window renewal theorem for non-arithmetic renewal processes with infinite mean, in the form given in \cite[Eq.~(1.1), p.~88]{Thorisson1987}.

\begin{theorem}\label{thm:mixed-renewal}
Assume that \(M\) is irreducible and stochastic, with \(m_{1,1}>0\), that \(F_i(0)=0\) for every \(i\), that \(F_1\) is non-arithmetic and satisfies \eqref{eq:F1-heavy-tail}, and that
\[
  \int_0^\infty\overline F_i(t)\,dt<\infty,
  \qquad i=2,\dots,K.
\]
Let \(z_1,\dots,z_K\) be real-valued Borel measurable functions, locally bounded on \([0,\infty)\), and extended by zero to \((-\infty,0)\). Assume that there exists \(a\in\R\) such that
\begin{equation}\label{eq:z1-decomp-renewal}
  z_1(t)=a\overline F_1(t)+r_1(t),
  \qquad t\ge0.
\end{equation}
Assume that each function among \(r_1,z_2,\dots,z_K\) satisfies
\begin{equation*}\tag{H}
  |h(t)|=o\bigl(\overline F_1(t)\bigr),
  \qquad t\to\infty.
\end{equation*}
where \(h\) denotes the function under consideration.
Then the renewal system
\begin{equation}\label{eq:mixed-renewal-system-final}
  H_i(t)
  =
  z_i(t)
  +
  \sum_{j=1}^K
  \int_{[0,t]}H_j(t-s)\,dF_{i,j}(s),
  \qquad i=1,\dots,K,
\end{equation}
has a unique locally bounded Borel measurable solution \(H=(H_1,\dots,H_K)\). Moreover,
\begin{equation}\label{eq:mixed-renewal-limit-final}
  \lim_{t\to\infty}H_i(t)=a,
  \qquad i=1,\dots,K.
\end{equation}
\end{theorem}

\begin{proof}
We first establish non-explosion and the local finiteness estimate needed below. Fix \(\theta>0\), and set
\[
  \phi_j(\theta):=
  \int_{[0,\infty)}e^{-\theta s}\,dF_j(s),
  \qquad
  q_\theta:=\max_{1\le j\le K}\phi_j(\theta).
\]
Since \(F_j(0)=0\) for every \(j\) and the type space is finite,
\[
  q_\theta<1.
\]
The Markov--renewal construction and the independence of the lifetime variables give
\[
  \E_i\left(e^{-\theta S_n}\right)
  =
  \E_i\left[
    \prod_{r=0}^{n-1}\phi_{Y_r}(\theta)
  \right]
  \le q_\theta^n.
\]
Since \((S_n)_{n\ge0}\) is nondecreasing, the extended limit
\[
  S_\infty:=\lim_{n\to\infty}S_n
  \in[0,\infty]
\]
exists almost surely. With the convention \(e^{-\theta\infty}=0\),
\[
  e^{-\theta S_n}\longrightarrow e^{-\theta S_\infty}
  \qquad\text{almost surely}.
\]
Since \(0\le e^{-\theta S_n}\le1\), the dominated convergence theorem gives
\[
  \E_i\left(e^{-\theta S_\infty}\right)
  =
  \lim_{n\to\infty}\E_i\left(e^{-\theta S_n}\right)
  =0.
\]
Therefore
\[
  e^{-\theta S_\infty}=0
  \qquad \Pbb_i\text{-almost surely},
\]
which is equivalent to
\[
  S_n\longrightarrow\infty
  \qquad \Pbb_i\text{-almost surely}.
\]
For \(i=1,\dots,K\), define
\[
  \widehat H_i(t):=
  \E_i\left[
    \sum_{n\ge0}
    z_{Y_n}(t-S_n)\ind{S_n\le t}
  \right],
  \qquad t\ge0.
\]
With \(\theta>0\) and \(q_\theta<1\) fixed as above, let \(T>0\) and set
\[
  C_T:=
  \max_{1\le j\le K}\sup_{0\le u\le T}|z_j(u)|.
\]
Since
\[
  \ind{S_n\le t}
  \le e^{\theta(t-S_n)}
  \le e^{\theta T}e^{-\theta S_n},
  \qquad 0\le t\le T,
\]
we have
\[
\begin{split}
  \E_i\left[
    \sum_{n\ge0}
    |z_{Y_n}(t-S_n)|\ind{S_n\le t}
  \right]
  &\le
  C_Te^{\theta T}
  \sum_{n\ge0}\E_i\left(e^{-\theta S_n}\right)\\
  &\le
  C_Te^{\theta T}\sum_{n\ge0}q_\theta^n
  =
  \frac{C_Te^{\theta T}}{1-q_\theta}
  <\infty.
\end{split}
\]
Thus the series defining \(\widehat H_i(t)\) is absolutely integrable. Since the corresponding expected partial sums are Borel measurable, \(\widehat H_i\) is Borel measurable, and the preceding bound shows that it is locally bounded. Conditioning on the first holding time and the first transition gives
\[
\begin{split}
  \widehat H_i(t)
  &=
  z_i(t)
  +
  \sum_{j=1}^K m_{i,j}\int_{[0,t]}
  \E_j\left[
    \sum_{n\ge0}
    z_{Y_n}(t-s-S_n)\ind{S_n\le t-s}
  \right]dF_i(s)\\
  &=
  z_i(t)
  +
  \sum_{j=1}^K
  \int_{[0,t]}\widehat H_j(t-s)\,dF_{i,j}(s).
\end{split}
\]
Hence \(\widehat H=(\widehat H_1,\dots,\widehat H_K)\) is a locally bounded Borel measurable solution of \eqref{eq:mixed-renewal-system-final}.

To prove uniqueness, let \(H\) and \(\widetilde H\) be two locally bounded Borel measurable solutions of \eqref{eq:mixed-renewal-system-final}, and set
\[
  D_i(t):=H_i(t)-\widetilde H_i(t),
  \qquad i=1,\dots,K.
\]
Then \(D=(D_1,\dots,D_K)\) satisfies the homogeneous system
\[
  D_i(t)
  =
  \sum_{j=1}^K
  \int_{[0,t]}D_j(t-s)\,dF_{i,j}(s).
\]
Since \(S_1=\tau_0\) and
\[
  \Pbb_i\{Y_1=j,\ S_1\in ds\}
  =
  m_{i,j}\,dF_i(s)
  =
  dF_{i,j}(s),
\]
the preceding equation is equivalently
\[
  D_i(t)
  =
  \E_i\left[
    D_{Y_1}(t-S_1)\ind{S_1\le t}
  \right].
\]
Repeated application of this identity, using the Markov--renewal property after each transition, gives, for every \(n\ge1\),
\[
  D_i(t)
  =
  \E_i\left[
    D_{Y_n}(t-S_n)\ind{S_n\le t}
  \right].
\]
Fix \(T>0\) and set
\[
  M_T:=
  \max_{1\le j\le K}\sup_{0\le u\le T}|D_j(u)|<\infty.
\]
For \(0\le t\le T\),
\[
  |D_i(t)|
  \le
  M_T\Pbb_i\{S_n\le t\}
  \le
  M_T\Pbb_i\{S_n\le T\}.
\]
Since \(S_n\to\infty\) almost surely,
\[
  \Pbb_i\{S_n\le T\}\longrightarrow0,
  \qquad n\to\infty.
\]
Letting \(n\to\infty\) in the preceding bound yields \(D_i(t)=0\). Since \(T>0\) is arbitrary, \(D_i(t)=0\) for every \(t\ge0\) and every \(i\in\Kk\). Consequently, \(\widehat H\) is the unique locally bounded Borel measurable solution of \eqref{eq:mixed-renewal-system-final}. Denoting this solution by \(H\), we obtain
\begin{equation}\label{eq:renewal-representation-final}
  H_i(t)
  =
  \E_i\left[
    \sum_{n\ge0}
    z_{Y_n}(t-S_n)\ind{S_n\le t}
  \right],
  \qquad t\ge0,\quad i=1,\dots,K.
\end{equation}

We next analyse the return cycles to type \(1\). Since \(M\) is irreducible on the finite state space \(\Kk\) and \(m_{1,1}>0\), there exist \(r\in\N\) and \(p_0>0\) such that
\[
  \inf_{i\in\Kk}\Pbb_i\{\sigma_1\le r\}\ge p_0.
\]
By the Markov property,

\begin{equation}\label{eq:return-exp-tail-final}
  \Pbb_i\{\sigma_1>nr\}\le(1-p_0)^n,\qquad n\ge0,\quad i\in\Kk.
\end{equation}
Consequently,
\[
  \Pbb_i\{\sigma_1=\infty\}=\lim_{n\to\infty}\Pbb_i\{\sigma_1>nr\}=0,\qquad \E_i(\sigma_1)=\sum_{m\ge0}\Pbb_i\{\sigma_1>m\}\le r\sum_{n\ge0}\Pbb_i\{\sigma_1>nr\}\le\frac{r}{p_0}<\infty.
\]

Let
\[
  \overline F_C(t):=\Pbb_1\{C>t\}.
\]
We claim that
\begin{equation}\label{eq:cycle-tail-final}
  \overline F_C(t)
  \sim\overline F_1(t)
  \sim c_1t^{-\gamma},
  \qquad t\to\infty.
\end{equation}
Write
\[
  C=\tau_0+E,
  \qquad
  E:=\sum_{n=1}^{\sigma_1-1}\tau_n,
\]
with the convention that \(E=0\) when \(\sigma_1=1\). Before the return to type \(1\), the chain visits only types in \(\{2,\dots,K\}\). Set
\[
  \mu_i:=\int_0^\infty\overline F_i(s)\,ds,\quad i=2,\dots,K,
  \qquad
  \mu_*:=\max_{2\le i\le K}\mu_i<\infty.
\]
Since \(M\) is stochastic, it is the transition matrix of the embedded type chain \((Y_n)_{n\ge0}\). By construction, conditionally on this chain, \(\tau_n\) has distribution \(F_{Y_n}\). Moreover, \(Y_n\ne1\) for \(1\le n<\sigma_1\). Therefore, by conditional expectation and Tonelli's theorem,
\[
  \E_1(E)
  =
  \E_1\left[\sum_{n=1}^{\sigma_1-1}\tau_{n}\right]
  \le
  \mu_*\E_1(\sigma_1-1)
  <\infty.
\]
On the event \(\{E>t\}\), one has \(t\le E\), and therefore
\[
  t\Pbb_1\{E>t\}
  =
  \E_1\bigl(t\ind{E>t}\bigr)
  \le
  \E_1\bigl(E\ind{E>t}\bigr).
\]
Moreover, \(\ind{E>t}\to0\) almost surely and
\[
  0\le E\ind{E>t}\le E.
\]
Since \(\E_1(E)<\infty\), dominated convergence yields
\[
  \E_1\bigl(E\ind{E>t}\bigr)\longrightarrow0.
\]
Consequently,
\[
  t\Pbb_1\{E>t\}\longrightarrow0,
\]
or equivalently, \(\Pbb_1\{E>t\}=o(t^{-1})\). Since \(0<\gamma<1\) and \(\overline F_1(t)\sim c_1t^{-\gamma}\), we have \(t^{-1}=o(\overline F_1(t))\), and therefore
\[
  \Pbb_1\{E>t\}
  =
  o\bigl(\overline F_1(t)\bigr).
\]

Since \(E\ge0\),
\[
  \overline F_C(t)\ge\overline F_1(t).
\]
For \(\delta\in(0,1)\),
\[
  \{C>t\}
  \subseteq
  \{\tau_0>(1-\delta)t\}\cup\{E>\delta t\},
\]
and therefore
\[
  \overline F_C(t)
  \le
  \overline F_1((1-\delta)t)+\Pbb_1\{E>\delta t\}.
\]
Dividing by \(\overline F_1(t)\) gives
\[
  \limsup_{t\to\infty}
  \frac{\overline F_C(t)}{\overline F_1(t)}
  \le(1-\delta)^{-\gamma}.
\]
Letting \(\delta\downarrow0\) proves \eqref{eq:cycle-tail-final}. In particular,
\[
  \E_1(C)=\infty.
\]

The distribution of \(C\) is non-arithmetic. Indeed, suppose that
\[
  \Pbb_1\{C\in h\N_0\}=1
\]
for some \(h>0\), and let \(\mathcal A:=\{Y_1=1\}\). Since \(m_{1,1}>0\),
\[
  \Pbb_1(\mathcal A)=m_{1,1}>0.
\]
On \(\mathcal A\), one has \(\sigma_1=1\) and \(C=\tau_0\). Since \(\tau_0\) is independent of \(\mathcal A\),
\[
\begin{split}
  0
  &=
  \Pbb_1\{C\notin h\N_0\}\ge
  \Pbb_1\bigl(\mathcal A,\tau_0\notin h\N_0\bigr)=
  \Pbb_1(\mathcal A)\Pbb_1\{\tau_0\notin h\N_0\}.
\end{split}
\]
It follows that \(F_1\) is arithmetic, contradicting the hypothesis.

Let \(h\) be a locally bounded Borel function on \([0,\infty)\), extended by zero to \((-\infty,0)\), and suppose that \(h\) satisfies \textup{(H)}. We claim that
\begin{equation}\label{eq:basic-negligible-final}
  (U_C*h)(t):=
  \int_{[0,t]}h(t-s)\,U_C(ds)
  \longrightarrow0,
  \qquad t\to\infty.
\end{equation}
Indeed, by condition \textup{(H)} and \eqref{eq:cycle-tail-final},
\[
  \frac{|h(t)|}{\overline F_C(t)}
  =
  \frac{|h(t)|}{\overline F_1(t)}
  \frac{\overline F_1(t)}{\overline F_C(t)}
  \longrightarrow0.
\]
Hence Lemma~\ref{lem:infinite-mean-renewal-negligibility}, applied to the return-cycle distribution \(F_C\), proves \eqref{eq:basic-negligible-final}. Moreover, \(U_C*h\) is bounded on \([0,\infty)\): it is bounded on compact intervals because \(U_C\) is finite on compact intervals and \(h\) is locally bounded, and it is bounded for large arguments because it converges to zero.

For the process started from type \(1\), define
\begin{equation}\label{eq:cycle-reward-final}
  q(t):=
  \E_1\left[
    \sum_{n=0}^{\sigma_1-1}
    z_{Y_n}(t-S_n)\ind{S_n\le t}
  \right].
\end{equation}

For \(n\ge0\), let
\[
  \mathcal G_n
  :=
  \sigma\bigl(
    Y_0,\ldots,Y_n,
    \tau_0,\ldots,\tau_{n-1}
  \bigr),
\]
with the list of holding times understood to be empty when \(n=0\), and
let
\[
  \mathcal G_\infty
  :=
  \sigma\left(\bigcup_{n\ge0}\mathcal G_n\right).
\]
Then each \(\sigma_k\) is a stopping time with respect to
\((\mathcal G_n)_{n\ge0}\). We write
\[
  \mathcal G_{\sigma_k}
  :=
  \left\{
    A\in\mathcal G_\infty:
    A\cap\{\sigma_k\le n\}\in\mathcal G_n
    \text{ for every }n\ge0
  \right\}
\]
for the corresponding stopping-time sigma-field.

Let
\[
  R_k:=S_{\sigma_k},
  \qquad k\ge0.
\]
Since \(Y_{\sigma_k}=1\), the strong Markov property at \(\sigma_k\)
implies that the increments \(R_{k+1}-R_k\) are independent and
identically distributed with law \(F_C\). Hence
\[
  U_C(\cdot)
  =
  \sum_{k\ge0}\Pbb_1\{R_k\in\cdot\}.
\]

For \(k\ge0\), define the reward accumulated during the \(k\)-th return
cycle by
\[
  Q_k(t)
  :=
  \sum_{n=\sigma_k}^{\sigma_{k+1}-1}
  z_{Y_n}(t-S_n)\ind{S_n\le t}.
\]
The cycle intervals partition the nonnegative integers, and the local
absolute-integrability estimate established above gives
\[
\begin{split}
  \sum_{k\ge0}\E_1|Q_k(t)|
  &\le
  \E_1\left[
    \sum_{n\ge0}
    |z_{Y_n}(t-S_n)|\ind{S_n\le t}
  \right]
  <\infty.
\end{split}
\]
Consequently,
\[
  H_1(t)=\sum_{k\ge0}\E_1[Q_k(t)].
\]

Set
\[
  \rho_k:=\sigma_{k+1}-\sigma_k,\qquad
  \widetilde Y_m^{(k)}:=Y_{\sigma_k+m},\qquad
  \widetilde S_m^{(k)}:=S_{\sigma_k+m}-R_k.
\]
Then
\[
\begin{split}
  Q_k(t)
  &=
  \ind{R_k\le t}
  \sum_{m=0}^{\rho_k-1}
  z_{\widetilde Y_m^{(k)}}
  \bigl((t-R_k)-\widetilde S_m^{(k)}\bigr)
  \ind{\widetilde S_m^{(k)}\le t-R_k}.
\end{split}
\]
Conditional on \(\mathcal G_{\sigma_k}\), the shifted sequence
\[
  \left(
    \widetilde Y_m^{(k)},
    \widetilde S_m^{(k)}
  \right)_{m\ge0}
\]
has the same law as the original Markov--renewal sequence under
\(\Pbb_1\). Since \(R_k\) is \(\mathcal G_{\sigma_k}\)-measurable, the
definition of \(q\) yields
\[
  \E_1\left[
    Q_k(t)\mid\mathcal G_{\sigma_k}
  \right]
  =
  q(t-R_k)\ind{R_k\le t}.
\]
Therefore,
\begin{equation}\label{eq:H1-cycle-final}
H_1(t)=
  \sum_{k\ge0}
  \E_1\left[
    q(t-R_k)\ind{R_k\le t}
  \right]=
  \sum_{k\ge0}
  \int_{[0,t]}q(t-s)\,\Pbb_1\{R_k\in ds\}=
  \int_{[0,t]}q(t-s)\,U_C(ds).
\end{equation}

The term corresponding to \(n=0\) in \eqref{eq:cycle-reward-final} is
\[
  z_1(t)=a\overline F_1(t)+r_1(t).
\]
For the remaining terms, the definition of \(\sigma_1\) as the first
return time to type \(1\) implies that
\[
  Y_n\in\{2,\dots,K\},
  \qquad 1\le n<\sigma_1.
\]
Accordingly, for \(j=2,\dots,K\), define the finite measure
\[
  \ell_j(B):=
  \E_1\left[
    \sum_{n=1}^{\sigma_1-1}
    \ind{Y_n=j,\ S_n\in B}
  \right]
\]
for every Borel set \(B\subset[0,\infty)\). Each \(\ell_j\) is finite,
since
\[
  \ell_j([0,\infty))
  \le
  \E_1(\sigma_1)
  <\infty.
\]

For every nonnegative Borel function \(g\), the definition of the mean
counting measure \(\ell_j\) and Tonelli's theorem give
\[
  \int_{[0,\infty)}g(s)\,\ell_j(ds)
  =
  \E_1\left[
    \sum_{n=1}^{\sigma_1-1}
    g(S_n)\ind{Y_n=j}
  \right].
\]
By applying this identity to the positive and negative parts, it also
holds for every Borel function \(g\) satisfying
\[
  \int_{[0,\infty)}|g(s)|\,\ell_j(ds)<\infty.
\]

For fixed \(t\ge0\), set
\[
  g_{j,t}(s)
  :=
  z_j(t-s)\ind{0\le s\le t}.
\]
Since \(z_j\) is locally bounded and \(\ell_j\) is finite,
\[
\begin{split}
  \int_{[0,\infty)}|g_{j,t}(s)|\,\ell_j(ds)
  &\le
  \sup_{0\le u\le t}|z_j(u)|\,
  \ell_j([0,t])
  <\infty.
\end{split}
\]
Consequently,
\[
\begin{split}
  &\E_1\left[
    \sum_{n=1}^{\sigma_1-1}
    z_j(t-S_n)\ind{Y_n=j,\ S_n\le t}
  \right]\\
  &\qquad=
  \int_{[0,t]}z_j(t-s)\,\ell_j(ds).
\end{split}
\]
Summing over \(j=2,\dots,K\), we obtain
\begin{equation}\label{eq:q-decomposition-final}
  q(t)
  =
  a\overline F_1(t)+r_1(t)
  +
  \sum_{j=2}^K
  \int_{[0,t]}z_j(t-s)\,\ell_j(ds).
\end{equation}
Substituting \eqref{eq:q-decomposition-final} into \eqref{eq:H1-cycle-final} and applying Fubini's theorem gives
\begin{equation}\label{eq:H1-conv-decomp-final}
\begin{split}
  H_1(t)
  &=
  a\int_{[0,t]}\overline F_1(t-s)\,U_C(ds)
  +(U_C*r_1)(t)+
  \sum_{j=2}^K
  \int_{[0,t]}(U_C*z_j)(t-s)\,\ell_j(ds).
\end{split}
\end{equation}
For fixed \(t\), the required absolute integrability follows from
\[
  \sup_{0\le u\le t}|z_j(u)|\,
  U_C([0,t])\,\ell_j([0,t])
  <\infty.
\]
By \eqref{eq:basic-negligible-final},
\[
  (U_C*r_1)(t)\longrightarrow0,
  \qquad
  (U_C*z_j)(t)\longrightarrow0.
\]
Each \(U_C*z_j\) is globally bounded. Hence, after extending it by zero to \((-\infty,0)\), dominated convergence with respect to the finite measure \(\ell_j\) yields
\[
  \int_{[0,t]}(U_C*z_j)(t-s)\,\ell_j(ds)
  \longrightarrow0.
\]
Therefore
\begin{equation}\label{eq:H1-main-reduced-final}
  H_1(t)
  =
  a\int_{[0,t]}\overline F_1(t-s)\,U_C(ds)+o(1).
\end{equation}

Set
\[
  d(t):=\overline F_1(t)-\overline F_C(t).
\]
By \eqref{eq:cycle-tail-final},
\[
  |d(t)|=o\bigl(\overline F_C(t)\bigr).
\]
Lemma~\ref{lem:infinite-mean-renewal-negligibility}, applied to \(C\) and \(d\), gives
\[
  (U_C*d)(t)\longrightarrow0.
\]
On the other hand, the elementary renewal identity \eqref{eq:renewal-tail-identity-general}, applied to \(C\), gives
\[
  \int_{[0,t]}\overline F_C(t-s)\,U_C(ds)=1.
\]
Consequently,
\begin{equation}\label{eq:F1-tail-renewal-final}
  \int_{[0,t]}\overline F_1(t-s)\,U_C(ds)
  =
  1+(U_C*d)(t)
  \longrightarrow1.
\end{equation}
Combining \eqref{eq:H1-main-reduced-final} and \eqref{eq:F1-tail-renewal-final}, we obtain
\[
  H_1(t)\longrightarrow a.
\]

Finally, fix \(i\in\{2,\dots,K\}\), and under \(\Pbb_i\) set
\[
  T_i:=S_{\sigma_1}.
\]
By \eqref{eq:return-exp-tail-final},
\[
  \Pbb_i\{\sigma_1<\infty\}=1,
  \qquad
  \E_i(\sigma_1)<\infty,
\]
and hence \(T_i<\infty\) almost surely. Define
\[
  b_i(t):=
  \E_i\left[
    \sum_{n=0}^{\sigma_1-1}
    z_{Y_n}(t-S_n)\ind{S_n\le t}
  \right].
\]
For \(j=2,\dots,K\), define
\[
  \nu_{i,j}(B):=
  \E_i\left[
    \sum_{n=0}^{\sigma_1-1}
    \ind{Y_n=j,\ S_n\in B}
  \right]
\]
for every Borel set \(B\subset[0,\infty)\). Then
\(
  \nu_{i,j}([0,\infty))
  \le
  \E_i(\sigma_1)
  <\infty
\)
and
\[
  b_i(t)
  =
  \sum_{j=2}^K
  \int_{[0,t]}z_j(t-s)\,\nu_{i,j}(ds).
\]
By condition \textup{(H)} and \(\overline F_1(t)\to0\), each \(z_j\), \(j\ge2\), tends to zero. Since each \(z_j\) is locally bounded, it is therefore bounded on \([0,\infty)\). Dominated convergence gives
\begin{equation}\label{eq:prehit-zero-final}
  b_i(t)\longrightarrow0.
\end{equation}

Extend \(H_1\) by zero to \((-\infty,0)\). Splitting \eqref{eq:renewal-representation-final} at the Markov-chain stopping time \(\sigma_1\) and using the strong Markov property of the Markov--renewal sequence gives
\begin{equation}\label{eq:hit-decomp-final}
  H_i(t)
  =
  b_i(t)
  +
  \E_i\left[
    H_1(t-T_i)\ind{T_i\le t}
  \right].
\end{equation}
Since \(H_1\) is locally bounded and converges to \(a\), it is globally bounded. Since \(T_i<\infty\) almost surely,
\[
  H_1(t-T_i)\ind{T_i\le t}
  \longrightarrow a
  \qquad\text{almost surely}.
\]
Dominated convergence therefore gives
\[
  \E_i\left[
    H_1(t-T_i)\ind{T_i\le t}
  \right]
  \longrightarrow a.
\]
Together with \eqref{eq:prehit-zero-final} and \eqref{eq:hit-decomp-final}, this proves
\[
  H_i(t)\longrightarrow a,
  \qquad i=1,\dots,K.
\]
\end{proof}

\section{Examples}\label{sec:examples}

We give explicit examples of non-arithmetic lifetime distributions for type \(1\) satisfying \eqref{eq:F1-heavy-tail}. For types \(2,\dots,K\), the non-arithmeticity assumption and condition \eqref{eq:finite-type-tails} are satisfied, for example, by bounded positive absolutely continuous, exponential, gamma, Weibull, and lognormal lifetimes, as well as by Pareto-type lifetimes with tail exponent greater than one.

\begin{example}\label{ex:AA-type-one-lifetime}
Let \(0<\gamma<1\). Let \(Z_\gamma\) be a positive \(\gamma\)-stable random variable with Laplace transform
\[
  \E\!\left(e^{-sZ_\gamma}\right)=e^{-s^\gamma},
  \qquad s\ge0,
\]
and let \(Y\) be an exponential random variable with rate \(\vartheta>0\), independent of \(Z_\gamma\). Define
\[
  T^{(1)}:=Z_\gamma Y^{1/\gamma}.
\]
Conditioning on \(Y\), we obtain
\[
  \E\!\left(e^{-sT^{(1)}}\right)
  =\E\!\left(e^{-s^\gamma Y}\right)
  =\frac{\vartheta}{\vartheta+s^\gamma}.
\]
Thus \(T^{(1)}\) has the Mittag--Leffler waiting-time distribution with parameters \((\gamma,\vartheta)\); see \cite[(1.3)--(1.6) and (2.41)--(2.42), with \(\delta=1\)]{CahoyPolito2013}. We use the notation
\[
  E_{a,b}(z):=\sum_{n=0}^{\infty}\frac{z^n}{\Gamma(an+b)},
  \qquad E_a:=E_{a,1}.
\]
The survival function and density of \(T^{(1)}\) are
\[
  \Pbb\{T^{(1)}>t\}=E_\gamma(-\vartheta t^\gamma),
  \qquad
  F_1'(t)=\vartheta t^{\gamma-1}E_{\gamma,\gamma}(-\vartheta t^\gamma),
  \qquad t>0.
\]
After the scaling \(t\mapsto\vartheta^{1/\gamma}t\), the asymptotic formulas in \cite[(4.5)--(4.6)]{MainardiGorenfloVivoli2005}, together with Euler's reflection formula, give
\[
  \Pbb\{T^{(1)}>t\}
  \sim\frac{1}{\vartheta\Gamma(1-\gamma)}t^{-\gamma},
  \qquad
  F_1'(t)
  \sim\frac{\gamma}{\vartheta\Gamma(1-\gamma)}t^{-1-\gamma},
  \qquad t\to\infty.
\]
Therefore \(F_1\) is absolutely continuous and hence non-arithmetic, and \eqref{eq:F1-heavy-tail} holds with
\[
  c_1=\frac{1}{\vartheta\Gamma(1-\gamma)}.
\]
\end{example}

\begin{example}\label{ex:pareto-lomax}
Let \(0<\gamma<1\). The Pareto distribution with scale \(x_0>0\) and shape \(\gamma\) has survival function
\[
  \overline F_1(t)=
  \begin{cases}
    1, & 0\le t<x_0,\\
    \left(\dfrac{x_0}{t}\right)^\gamma, & t\ge x_0,
  \end{cases}
\]
and density
\[
  F_1'(t)=\gamma x_0^\gamma t^{-1-\gamma}\ind{t\ge x_0},
  \qquad t>0.
\]
Therefore \(F_1\) is absolutely continuous and hence non-arithmetic, and \eqref{eq:F1-heavy-tail} holds with \(c_1=x_0^\gamma\).

The Lomax distribution with scale \(\vartheta>0\) and shape \(\gamma\) has survival function and density
\[
  \overline F_1(t)=\left(1+\frac{t}{\vartheta}\right)^{-\gamma},
  \qquad
  F_1'(t)=\frac{\gamma}{\vartheta}
  \left(1+\frac{t}{\vartheta}\right)^{-1-\gamma},
  \qquad t>0.
\]
Since
\[
  \overline F_1(t)\sim\vartheta^\gamma t^{-\gamma},
  \qquad t\to\infty,
\]
\(F_1\) is non-arithmetic and \eqref{eq:F1-heavy-tail} holds with \(c_1=\vartheta^\gamma\).
\end{example}

\begin{example}\label{ex:burr-loglogistic}
Let \(a,b,\vartheta>0\), and consider the Burr XII distribution with survival function
\[
  \overline F_1(t)=
  \left(1+\left(\frac{t}{\vartheta}\right)^a\right)^{-b},
  \qquad t\ge0.
\]
Its density is
\[
  F_1'(t)=\frac{ab}{\vartheta}
  \left(\frac{t}{\vartheta}\right)^{a-1}
  \left(1+\left(\frac{t}{\vartheta}\right)^a\right)^{-b-1},
  \qquad t>0.
\]
If \(ab=\gamma\in(0,1)\), then
\[
  \overline F_1(t)\sim\vartheta^\gamma t^{-\gamma},
  \qquad
  F_1'(t)\sim\gamma\vartheta^\gamma t^{-1-\gamma},
  \qquad t\to\infty.
\]
Therefore \(F_1\) is absolutely continuous and hence non-arithmetic, and \eqref{eq:F1-heavy-tail} holds with \(c_1=\vartheta^\gamma\). The log-logistic distribution corresponds to \(b=1\); hence choosing \(a=\gamma\in(0,1)\) gives another instance of \eqref{eq:F1-heavy-tail}.
\end{example}

\section*{Acknowledgements} The authors are grateful to José Alfredo López-Mimbela and Ekaterina Todorova Kolkovska, both from the Centro de Investigación en Matemáticas (CIMAT), for helpful comments and suggestions. F.P.M. and J.H.R.G. are supported by FAPESP (grants 2023/13453-5 and 2025/03804-0, respectively).

\end{document}